\documentclass[draft]{amsart}

\usepackage{amssymb}
\usepackage{url}
\usepackage{amscd}
\usepackage[all]{xy}

\theoremstyle{plain}
\newtheorem{theorem}{Theorem}

\newtheorem{corollary}[theorem]{Corollary}
\newtheorem{lemma}[theorem]{Lemma}
\newtheorem{proposition}[theorem]{Proposition}
\newtheorem{definition}[theorem]{Definition}
\theoremstyle{remark}
\newtheorem{remark}{Remark}
\newtheorem{example}{Example}

\begin{document}

\date{}

\title{The $OP$ subgroup of a torsion-free LCA group}

\author{Aliakbar Alijani}

\address{Department of Mathematics,\\ Technical and Vocational University(TVU),\\ Tehran,\\ Iran}

\email{alijanialiakbar@gmail.com}

\thanks{}


\begin{abstract}
Let $G$ be a locally compact abelian (LCA) group. We denote by $G_{op}$, the intersection of all open pure subgroups of $G$, which we call the $OP$ subgroup of $G$. In this paper, we prove that the $OP$ subgroup of a  torsion-free LCA group $G$ is a non-zero pure subgroup of $G$.
\end{abstract}

\maketitle

%%%%%%%%%%%%%%%%%%%%%%%%%%%%%%%%%%%%%%%%%%%%%%%%%%%%%%%%%%%%%%%%%%%%%%%%%
% Macros
%%%%%%%%%%%%%%%%%%%%%%%%%%%%%%%%%%%%%%%%%%%%%%%%%%%%%%%%%%%%%%%%%%%%%%%%%

\newcommand\sfrac[2]{{#1/#2}}

\newcommand\cont{\operatorname{cont}}
\newcommand\diff{\operatorname{diff}}

%%%%%%%%%%%%%%%%%%%%%%%%%%%%%%%%%%%%%%%%%%%%%%%%%%%%%%%%%%%%%%%%%%%%%%%%%

\section{Introduction}
\newcommand{\stk}{\stackrel}
Let $\pounds$ denote the category of locally compact abelian (LCA) groups with continuous homomorphisms as morphisms. For a group $G\in \pounds$, the identity component subgroup, and the subgroup of all compact elements of $G$ are denoted by $G_{0}$ and $bG$, respectively. A subgroup $H$ of $G$ is called pure if $nH=H\cap nG$ for all positive integers $n$. Recall that $G_{0}$ and $bG$ are two important and well-known pure subgroups of $G$. In this paper, we produce a new and non-zero pure subgroup of a torsion-free group $G$. Let $G_{op}$ be the intersection of all open pure subgroups of $G\in\pounds$. Then $G_{op}$ is a closed subgroup which we call the $OP$ subgroup of $G$. In this paper, we focus on the OP subgroup of a torsion-free group. We show that if $G$ is a torsion-free group, then $G_{op}$ is a non-zero pure subgroup of $G$ (see Theorem \ref{13}).

The additive topological group of real numbers is denoted by $\Bbb R$ , $\Bbb Q$ is the group of rationales with the discrete topology, $\Bbb Z$ is the group of integers and $\Bbb Z(n)$ is the cyclic group of order $n$. For any group $G$ and $H$, ${\rm Hom}(G,H)$ is the group of all continuous homomorphisms from $G$ to $H$, endowed with the compact-open topology. The dual group of $G$ is $\hat{G}=Hom(G,\Bbb R/\Bbb Z)$. For more on locally compact abelian groups, see \cite{HR}.

%%%%%%%%%%%%%%%
\section{The $OP$ subgroup of a torsion-free LCA group}

In this section, we introduce the concept and study some properties of the $OP$ subgroup of a torsion-free LCA group.

\begin{definition}\cite{F}
A subgroup $H$ of a group $G$ is said to be a pure subgroup if $nH=H\bigcap nG$ for all positive integers $n$.
\end{definition}

Let $G\in \pounds$, and $G_{op}$ be the intersection of all open pure subgroups of $G$. Then $G_{op}$ is a closed subgroup of $G$.

\begin{definition}
The subgroup $G_{op}$ of a group $G\in \pounds$ is called the $OP$ subgroup of $G$.
\end{definition}

\begin{remark}\label{1}
Let $D$ be a discrete group. It is clear that $0$ is an open pure subgroup of $D$. So $D_{op}=0$. The converse is not true. See the Example \ref{2}.
\end{remark}

\begin{example}\label{2}
Let $I$ be an infinite set, and $G=\prod_{I} \Bbb Z(2)$. Set $H_{j}=\prod_{i\in I} X_{i}$ where $X_{j}=0$ and $X_{i}=\Bbb Z(2)$ for $i\neq j$. By Theorem 2 of \cite{AR}, $H_{j}$ is a pure subgroup of $G$ for all $j\in I$. On the other hand, $H_{j}$ is an open subgroup of $G$ for all $j\in I$. Hence $G_{op}\subseteq \bigcap_{j\in I} H_{j}=0$.
\end{example}

\begin{lemma}\label{3}
Let $G$ be a group in $\pounds$. Then $G_{0}\subseteq G_{op}$.
\end{lemma}

\begin{proof}
By Theorem 7.8 of \cite{HR}, $G_{0}$ is the intersection of all open subgroups of $G$. So $G_{0}\subseteq G_{op}$.
\end{proof}

\begin{corollary}\label{4}
Let $C$ be a connected group. Then $C=C_{op}$.
\end{corollary}

\begin{proof}
It is clear by Lemma \ref{3}.
\end{proof}

\begin{definition}\cite{AR}
A group $G\in \pounds$ is said to be pure-simple if and only if $G$ contains no nontrivial closed pure subgroups.
\end{definition}

\begin{example}\label{5}
By \cite[Theorem~1]{AR}, $\widehat{\Bbb Q/\Bbb Z}$ is a pure-simple group. So $(\widehat{\Bbb Q/\Bbb Z})_{op}=\widehat{\Bbb Q/\Bbb Z}$. %On the other hand, there exists a non-zero morphism $f:J_{p}\to \Bbb Q/\Bbb Z$. It is clear that $f$ do not preserve OP subgroups.
\end{example}

The Example \ref{5} shows that the converse of Corollary \ref{4} need not to be true.

\begin{remark}\label{6}
Let $H$ be an open pure subgroup of $G$. It is clear that every open pure subgroups of $H$ is open pure in $G$. So $G_{op}\subseteq H_{op}$.
\end{remark}

\begin{lemma}\label{7}
Let $G\in \pounds$ be a torsion-free group and $H$, an open pure subgroup of $G$. Then $G_{op}=H_{op}$.
\end{lemma}

\begin{proof}
By Remark \ref{6}, $G_{op}\subseteq H_{op}$. Now, let $x\in H_{op}$ and $X$ is an arbitrary open pure subgroup of $G$. An easy calculation shows that $X\cap H$ is an open pure subgroup of $H$. Hence $x\in X$. So $G_{op}\subseteq H_{op}$, and proof is complete.
\end{proof}
%\begin{ex}\label{3}
%Let $G$ be a LCA group such that $G_{op}=0$. By Remark \ref{1}, $(G\times \Bbb Z)_{op}\subseteq G_{op}=0$. But, $G\times \Bbb Z$ is not discrete.
%\end{ex}

%\begin{lem}\label{4}
%Let $C$ be a connected group. Then, $C=C_{op}$.
%\end{lem}

%\begin{proof}
%It is sufficient to consider that if $G\in \pounds$, then $G_{0}\subseteq G_{op}$ (see Theorem 7.8 of \cite{HR}).
%\end{proof}
%\begin{rem}\label{7}
%Let $H_{1}$ and $H_{2}$ be arbitrary pure open subgroups of $G_{1}$ and $G_{2}$, respectively. For an arbitrary positive integer $n$,
%\begin{align*}
% n(H_{1} \oplus H_{2})\nonumber
 %&=nH_{1}\oplus nH_{2}\\\nonumber
 %&=(H_{1}\cap nG_{1})\oplus (H_{2}\cap nG_{2})\\\nonumber
 %&=(H_{1}\oplus H_{2})\cap (nG_{1} \oplus nG_{2})
%\end{align*}
%So, $H_{1}\oplus H_{2}$ is a pure open subgroup of $G_{1}\oplus G_{2}$.
%\end{rem}
%\begin{lem}\label{14}
%Let $G_{1},G_{2}\in \pounds$. Then $(G_{1}\oplus G_{2})_{op}\subseteq (G_{1})_{op}\oplus (G_{2})_{op}$.
%\end{lem}

%\begin{proof}
%Let $(x,y)\in (G_{1}\oplus\times G_{2})_{op}$ and $H$ be a pure open subgroup of $G_{1}$. By Remark \ref{7}, $H\oplus G_{2}$ is a open pure subgroup of $G_{1}\oplus G_{2}$. Hence, $x\in (G_{1})_{op}$. Similarly, it can be concluded that $y\in (G_{2})_{op}$.
%\end{proof}

%\begin{cor}\label{15}
%Let $G_{1},G_{2}\in \pounds$. Then
%$$(G_{1})_{0}\oplus(G_{2})_{0}\subseteq(G_{1}\oplus G_{2})_{op}\subseteq (G_{1})_{op}\oplus (G_{2})_{op}$$
%\end{cor}

Let $G$ and $H$ be groups in $\pounds$, and $f:G\to H$ a morphism. Then it is not necessary that $f(G_{op})\subseteq H_{op}$. For example, consider a nonzero morphism $f:\widehat{\Bbb Q/\Bbb Z}\to \Bbb Q/\Bbb Z$.

\begin{lemma}\label{8}
Let $f:G\to H$ be a morphism such that $H$ is a torsion-free group. Then $f(G_{op})\subseteq H_{op}$.
\end{lemma}

\begin{proof}
Let $y\in G_{op}$ and $X$, be an open pure subgroup of $H$. It is clear that $f^{-1}(X)$ is an open subgroup of $G$. Now we show that $f^{-1}(X)$ is pure in $G$. Let $n$ be a positive integer, and $g\in f^{-1}(X)\cap nG$. Then $f(g)\in X$, and $g=ng_{1}$ for some $g_{1}\in G$. So $nf(g_{1})\in X$. Since $X$ is pure in $H$, $nf(g_{1})=nx_{1}$ for some $x_{1}\in X$. Since $H$ is torsion-free, it follows that $f(g_{1})\in X$. Hence $g\in nf^{-1}(X)$, and $f^{-1}(X)$ is pure in $G$. So $f(y)\in X$, and proof is complete.
\end{proof}

Let $G\in \pounds$. We denote by $G^{*}$, the minimal divisible extension of $G$. By \cite[~4.18.h]{HR}, $G^{*}$ is an LCA group containing $G$ as an open subgroup.

\begin{corollary}\label{9}
Let $G$ be a torsion-free group in $\pounds$. Then $G_{op}\subseteq G^{*}_{op}$.
\end{corollary}

\begin{proof}
It is sufficient to consider the inclusion map $i:G\hookrightarrow G^{*}$ and Lemma \ref{8}.
\end{proof}

\begin{corollary}\label{10}
Let $G$ be a torsion-free group in $\pounds$ such that $G_{op}=G$. Then $G^{*}_{op}=G^{*}$.
\end{corollary}

\begin{proof}
By Corollary \ref{9}, $G\subseteq G^{*}_{op}$. Since $G^{*}$ is torsion-free, an easy calculation shows that $G^{*}_{op}$ is divisible. It follows from the minimality $G^{*}$ that $G^{*}_{op}=G^{*}$.
\end{proof}

A morphism is called proper if it is open onto its image, and a short exact sequence $0\to A\stackrel{\phi}{\to} B\stackrel{\psi}{\to}C\to 0$ in $\pounds$ is said to be an extension of $A$ by $C$ if $\phi$ and $\psi$ are proper morphisms. Following \cite{FG1}, we let $Ext(C,A)$ denote the (discrete) group extensions of $A$ by $C$.
Our next goal is to prove that the $OP$ subgroup of a torsion-free LCA group is nonzero. To do this, we need the following results about $Ext$.

\begin{lemma}(\cite[page 222(B)]{F})\label{14}
Let $D$ be a discrete divisible group. Then $Ext(C,D)=0$ for all discrete groups $C$.
\end{lemma}

\begin{theorem}(\cite[Theorem~2.12]{FG1})\label{15}
Let $A$ and $C$ be groups in $\pounds$. Then $Ext(C,A)\cong Ext(\hat{A},\hat{C})$.
\end{theorem}

\begin{proposition}(\cite[Proposition~2.17(f)]{FG1})\label{16}
Let $D$ be a discrete group. Then ${\rm Ext}(\hat{\Bbb Z},D)\cong D$.
\end{proposition}

The exact sequences (1) and (2) of the following Theorem establish a closed connection between $Hom$ and $Ext$ in $\pounds$.
\begin{theorem}(\cite[Corollary~2.10]{FG2})\label{11}
Let $G\in \pounds$ and $0 \to A \to B \to C \to 0$ be an extension in $\pounds$. Then, the following sequences are exact:
\begin{enumerate}
\item $0\to Hom(C,G)\to Hom(B,G)\to Hom(A,G)\to Ext(C,G)\to Ext(B,G)\to Ext(A,G)\to 0$
\item $0\to Hom(G,A)\to Hom(G,B)\to Hom(G,C)\to Ext(G,A)\to Ext(G,B)\to Ext(G,C)\to 0$
\end{enumerate}
\end{theorem}

\begin{lemma}\label{12}
Let $G\in \pounds$ be a non discrete, totally disconnected, torsion free group. Then $Hom(\widehat{\Bbb Q/\Bbb Z},G)\neq 0$.
\end{lemma}

\begin{proof}
By Theorem 24.30 of \cite{HR}, $G$ contains a compact open subgroup $K$. Consider the two exact sequences $0\to K\hookrightarrow G\to G/K\to 0$ and $0\to \widehat{\Bbb Q/\Bbb Z}\to \hat{\Bbb Q}\to \hat{\Bbb Z}\to 0$. By Theorem \ref{11}, we have the following exact sequences
\begin{equation}
0\to Hom(\widehat{\Bbb Q/\Bbb Z},K)\to Hom(\widehat{\Bbb Q/\Bbb Z},G)\to ...
\end{equation}
\begin{equation}
...\to Hom(\widehat{\Bbb Q/\Bbb Z},K)\to Ext(\hat{\Bbb Z},K)\to Ext(\hat{\Bbb Q},K)
\end{equation}
Let $Hom(\widehat{\Bbb Q/\Bbb Z},G)=0$. By (2.1), $Hom(\widehat{\Bbb Q/\Bbb Z},K)=0$. On the other hand, $Ext(\hat{\Bbb Q},K)\cong Ext(\hat{K},\Bbb Q)=0$ (see Lemma \ref{14} and Theorem \ref{15}). By (2.2) and Proposition \ref{16}, $K\cong Ext(\hat{\Bbb Z},K)=0$ which is a contradiction. Hence $Hom(\widehat{\Bbb Q/\Bbb Z},G)\neq0$, and proof is complete.
\end{proof}

\begin{theorem}\label{13}
Let $G\in \pounds$ be a non discrete, torsion free group. Then $G_{op}\neq 0$.
\end{theorem}

\begin{proof}
Assume to contrary, $G_{op}=0$. Then, by Lemma \ref{3}, $G$ is a totally disconnected group. By Lemma \ref{12}, there exists a nonzero morphism $f:\widehat{\Bbb Q/\Bbb Z}\to G$. By Lemma \ref{8}, $f((\widehat{\Bbb Q/\Bbb Z})_{op})\subseteq G_{op}=0$. On the other hand, by Example \ref{5}, $(\widehat{\Bbb Q/\Bbb Z})_{op}=\widehat{\Bbb Q/\Bbb Z}$. It follows that $f$ is a zero morphism which is a contradiction.
\end{proof}

%\begin{lem}
%Let $G$ be a compact group. Then $H$ is a pure open subgroup of $G$ if and only if $(\hat{G},H)$ is a pure open subgroup of $\hat{G}$.
%\end{lem}

%\begin{cor}
%Let $G$ be a compact group. Then $(\hat{G})_{op}\subseteq(\hat{G},G_{op})$.
%\end{cor}

%\begin{proof}
%It is clear by
%\end{proof}

%\begin{cor}
%Let $G$ be a group in $\pounds$. Then $G_{op}\subseteq bG$.
%\end{cor}

%\begin{proof}
%By Theorem 24.30 of \cite {HR}, $G=\Bbb R^{n}\oplus H$ where $H$ contains a compact open subgroup $H$. By Lemma 2.7 of \cite{AY}, $bH$ is an open pure subgroup of $H$. So $H_{op}\subseteq bH$.
%\end{proof}

\end{document}